\documentclass[leqno,12pt]{amsart} 
\usepackage{amssymb} 
\usepackage[all]{xy}
\usepackage{graphicx}
\theoremstyle{plane} 
\newtheorem{theorem}{\indent\sc Theorem}[section] 
\newtheorem{lemma}[theorem]{\indent\sc Lemma}

\newtheorem{proposition}[theorem]{\indent\sc Proposition}

\newtheorem{problem}[theorem]{\indent\sc Problem}
\newtheorem{main}[theorem]{\indent\sc Main Theorem}

\theoremstyle{definition}

\begin{document}

\title[second exterior power of tangent bundles of Fano fourfolds]{On the second exterior power of tangent bundles of Fano fourfolds with Picard number $\rho(X)\geqslant2$} 

\author[K. Yasutake]{Kazunori Yasutake} 

%%%%%%%%%%%%%%%%%%% ‹r' %%%%%%%%%%%%%%%%%%%

%%%%%%%%%%%% '˜ŽÒŠ'® %%%%%%%%%%%%%
\address{
Organization for the Strategic Coordination of Research and 
Intellectual Properties \endgraf
Meiji University \endgraf
Kanagawa 214-8571  \endgraf
Japan
}
\email{tz13008@meiji.ac.jp}
%%%%%%%%%%%%%%%%%%%%%%%%%%%%%%%%%%%%%%%%%

\maketitle

\begin{abstract}
In this paper, we classify Fano fourfolds  with Picard number greater than one such that the second exterior power of tangent bundles are numerically effective.
\end{abstract}

\section*{Introduction}
In the paper \cite{peternell2}, F. Campana and T. Peternell classify smooth projective threefolds $X$ such that the second exterior power of tangent bundle $\Lambda^2\mathcal{T}_X$ are numerically effective $($nef , for short$)$. 
Manifolds satisfying this condition are rare in manifolds whose the anti-canonical class $-K_X$ is nef.
Such threefolds are as follows.

\begin{theorem}[\cite{peternell2}, Theorem]\label{3-dim}
Let $X$ be a projective threefold with $\Lambda^2\mathcal{T}_X$ nef.
Then either $\mathcal{T}_X$ is nef or $X$ is one of the following.
\begin{enumerate}
\item $X$ is a blowing-up of $\mathbb{P}^3$ at a point.
\item $X$ is a del Pezzo threefold $($ that is, Fano threefold with index two$)$ of $\rho(X)=1$ except for those of degree one.
\end{enumerate}
\end{theorem}  

In this paper, as a first step of classifying fourfolds with the second exterior power of tangent bundle nef, we treat the case where X is Fano fourfolds with Picard number $\rho(X)$ greater than one.
We get the following theorem.      
\begin{main}
Let X be a Fano fourfold with $\rho(X)\geqslant2$.
Assume that $\Lambda^2\mathcal{T}_X$ is nef.
Then X is the blowing-up of $\mathbb{P}^4$ at a point unless $\mathcal{T}_X$ is nef .
\end{main}

There is a problem about the nefness of $\Lambda^q\mathcal{T}_X$ posed by F. Campana and T. Peternell.
\begin{problem}[\cite{peternell2}, Problem 6.4]
Let X be a Fano manifold. Assume that $\Lambda^q\mathcal{T}_X$ is nef on every extremal rational curve.
Is then $\Lambda^q\mathcal{T}_X$ already nef ?
\end{problem}

In the proof of main theorem we can see that the following theorem also holds.

\begin{theorem}
Let X be a Fano fourfold with $\rho(X)\geqslant2$.
Assume that $\Lambda^2\mathcal{T}_X$ is nef on every extremal rational curves in X.
Then $\Lambda^2\mathcal{T}_X$ is nef.
\end{theorem}  

\section*{Acknowledgements}
The author would like to express his gratitude to Professor Eiichi Sato for many useful discussions and much warm encouragement. 
\section*{Notation}
Throughout this paper we work over the complex number field $\mathbb{C}$.
We freely use the customary terminology in algebraic geometry.
We denote the Picard number of a variety $X$ by $\rho(X)$.  
We say that an extremal contraction is of $(m,n)$-type if the dimension of the exceptional locus is equal to $m$ and the dimension of the image of the exceptional locus is equal to $n$.
\\

\section{General Result}
Let $X$ be a smooth $n$-dimensional projective variety with $\Lambda^r\mathcal{T}_X$ nef for some integer $r$, 
$1\leqslant r \leqslant n$.
Since $\det(\Lambda^r\mathcal{T}_X)=-\displaystyle\binom{n-1}{r-1}K_X$ is also nef, the Kodaira dimension of $X$ is non-positive $\kappa(X)\leqslant 0$.
In this section, we classify the case where $\kappa(X)=0$.
Main theorem of this section is the following.
\begin{theorem}
Let X be a smooth n-dimensional projective variety with nef vector bundle $\Lambda^r\mathcal{T}_X$ $(1\leqslant r \leqslant n-1)$ and $\kappa(X)=0$. Then after taking some \'etale covering $f:\tilde{X}\rightarrow X$, $\tilde{X}$ is isomorphic to an abelian variety. 
\end{theorem}
\begin{proof}
If $r=1$, this is proved in \cite{peternell1}, Theorem 2.3.
Therefore we assume that $r\geqslant2$.
By the nefness of $-K_X$ and $\kappa(X)=0$, we know that the canonical line bundle $K_X$ is a torsion line bundle.
Therefore by taking some \'etale covering of $X$, we may assume that $K_X$ is trivial.
From the existence of K\"ahler-Einstein metric on X \cite{yau}, we know that $\mathcal{T}_X$ is H-semistable in the sense of Takemoto-Mumford with respect to any ample divisor H \cite{tsuji}, \cite{enoki}. On the other hand, $\Lambda^r\mathcal{T}_X$ is nef vector bundle with trivial determinant. 
Hence $\Lambda^r\mathcal{T}_X$ is a numerically flat vector bundle i.e. the dual vector bundle $(\Lambda^r\mathcal{T}_X)^{\vee}$ is also nef.
Therefore we can show that the second Chern class of $\Lambda^r\mathcal{T}_X$ are numerical trivial.
In particular we have 
\[c_1(\Lambda^r\mathcal{T}_X).H^{n-1}=-\displaystyle\binom{n-1}{r-1}K_X.H^{n-1}=0\] 
and 
\[c_2(\Lambda^r\mathcal{T}_X).H^{n-2}=\{\displaystyle\binom{d}{2}c_1^2(X)+\binom{n-2}{r-1}c_2(X)\}.H^{n-2}=0,\]
where $d=\dbinom{n-1}{r-1}$. 
By the triviality of $K_X$, we have the equality $c_2(X).H^{n-2}=0$. 
Since $\mathcal{T}_X$ is H-semistable, $c_1(X)=0$ and $c_2(X).H^{n-2}=0$, $\mathcal{T}_X$ is numerically flat, see \cite{nakayama} Theorem 4.1 in Section 4.
Therefore from theorem 2.3 in \cite{peternell1}, we have an \'etale covering $f:\tilde{X}\rightarrow X$ such that $\tilde{X}$ is isomorphic to an abelian variety.   
\end{proof}

By the theorem above, we have only to consider the case where $\kappa(X)=-\infty$.
This case is more difficult than the former one.
In the rest of this paper we treat the case where 
X is a Fano fourfold with $\rho(X)\geqslant2$.  

\section{Proof of Main Theorem}

At first, we consider the case where $X$ is obtained by a blowing-up of a smooth variety along a smooth subvariety.  
\begin{lemma}\label{blowup}
Let X be an n-dimensional Fano manifold. We assume that X is obtained by a blowing-up of a smooth manifold Y along a smooth subvariety Z.
If $\displaystyle \Lambda ^2 \mathcal{T}_X$ is nef, then Y is the projective space $\mathbb{P}^n$ and Z is a point.
\end{lemma}

\begin{proof}
At first, we consider the case where $\dim Z\geqslant 1$.
Let $E\cong \mathbb{P}_Z(N_{Z/Y}^{\vee})$ be the exceptional divisor of the blowing-up $\phi:X=Bl_Z(Y)\rightarrow Y$. 
Let F be a general fiber of $\phi|_E$. Then we have an exact sequence 
\[0\rightarrow N_{F/E}\rightarrow N_{F/X}\rightarrow N_{E/X}|_F \rightarrow 0.\]
Hence we get 
\[\Lambda^2N_{F/X}\rightarrow N_{E/X}|_F\otimes N_{F/E} \rightarrow 0.\]
Combining this with the homomorphism $\Lambda^2\mathcal{T}_X|_F\rightarrow \Lambda^2N_{F/X}\rightarrow 0$, we obtain the surjective homomorphism of vector bundles
 \[\Lambda^2\mathcal{T}_X|_F\rightarrow N_{E/X}|_F\otimes N_{F/E}\rightarrow 0.\]
We put $c=n-1-\dim Z$.   Since $F\cong \mathbb{P}^{c}$ and $N_{E/X}|_F\otimes N_{F/E}\cong\mathcal{O}_{F}^{\oplus c}(-1)$ is not nef, we know that $\displaystyle \Lambda ^2 \mathcal{T}_X$ is not nef.
  
  Next, we consider the case where $\dim Z=0$. 
 It follows from the argument in \cite{campana} Theorem 1.1.
  For the convenience, we write the argument.  
  In this case, there exists an extremal rational curve $C$ such that $E.C>0$.
  In fact, there is a curve $C'$ such that $E.C'>0$ because $E$ is effective and nontrivial. 
  Since X is Fano, $C'$ is a linear combination of extremal rational curves with non-negative real coefficient.
  Hence one of extremal rational curves $C$ satisfy that $E.C>0$. 
  We consider the extremal contraction map $\varphi:X\rightarrow W$ associated with $C$. 
  Then, nontrivial fibers are one dimensional. 
  Therefore from Theorem \cite{ando} and first part of this proof, $\varphi$ is conic bundle which turn to be $\mathbb{P}^1$-bundle by Lemma\ref{2}.
  Since $\varphi|_E: E\cong\mathbb{P}^{n-1}\rightarrow W$ is finite surjective map, $W$ is isomorphic to $\mathbb{P}^{n-1}$ by Theorem \ref{lazarsfeld}. 
We prove that E is a section of $\varphi$. We assume that $\varphi : E\rightarrow W\cong\mathbb{P}^{n-1}$ is ramified.
In this case, there is a line $l\subseteq W$ we have a Hirzebruch surface $S:=\varphi^{-1}(l)$ such that
  $E\cap S$ is a multisection of $\pi|_S: S\rightarrow l$ and $\phi|_S(S\cap E)$ is a point.
  This is a contradiction since the negative section is the only curve contracted to a point by a birational map on a Hirzebruch surface.    
Therefore $E$ is a section of $\varphi$.
  Hence $X$ is isomorphic to a projectivization of a rank 2 vector bundle $\mathbb{P}_{W}(\mathcal{E})$ such that the tautological divisor on $\mathbb{P}_{W}(\mathcal{E})$ is $E$.   
In this case we have an exact sequence 
\[ 0\rightarrow \mathcal{O}_X\rightarrow \mathcal{O}_X(E)\rightarrow \mathcal{O}_{E}(E)\cong\mathcal{O}_E(-1)\rightarrow 0.\] 
From this exact sequence, we have the following exact sequence by push forward;
  \[ 0\rightarrow \mathcal{O}_W\rightarrow \varphi_*(\mathcal{O}_X(E))\cong\mathcal{E}\rightarrow \mathcal{O}_W(-1)\rightarrow 0.\] 
Because W is a projective space, this exact sequence splits and we have that $X\cong\mathbb{P}_{\mathbb{P}^{n-1}}(\mathcal{O}_{\mathbb{P}^{n-1}}\oplus\mathcal{O}_{\mathbb{P}^{n-1}}(1))$ which is a blowing-up of $\mathbb{P}^n$ at a point. 
Finally, we show that $\Lambda^2\mathcal{T}_X$ is nef.
Let $\pi:\mathbb{P}_{\mathbb{P}^{n-1}}(\mathcal{O}_{\mathbb{P}^{n-1}}\oplus\mathcal{O}_{\mathbb{P}^{n-1}}(1))\rightarrow \mathbb{P}^{n-1}$ be the natural projection.
Then we have the following exact sequence
\[0\rightarrow \mathcal{T}_{\pi}\rightarrow \mathcal{T}_X\rightarrow \pi^*\mathcal{T}_{\mathbb{P}^{n-1}}\rightarrow 0. \]
Therefore we get the exact sequence
\[0\rightarrow \mathcal{T}_{\pi}\otimes\pi^*\mathcal{T}_{\mathbb{P}^{n-1}}\rightarrow \Lambda^2\mathcal{T}_X\rightarrow \pi^*\Lambda^2\mathcal{T}_{\mathbb{P}^{n-1}}\rightarrow 0. \]
Let $\xi$ be the tautological line bundle on $\mathbb{P}_{\mathbb{P}^{n-1}}(\mathcal{O}_{\mathbb{P}^{n-1}}\oplus\mathcal{O}_{\mathbb{P}^{n-1}}(1))$.
Since $\mathcal{T}_{\pi}\otimes\pi^*\mathcal{T}_{\mathbb{P}^{n-1}}\cong2\xi\otimes\pi^*\mathcal{T}_{\mathbb{P}^{n-1}}(-1)$ is nef, $\Lambda^2\mathcal{T}_X$ is also nef.
\end{proof}
\begin{theorem}[\cite{lazarsfeld}, Theorem 4.1]\label{lazarsfeld}
Let X be a smooth projective variety of dimension $\geqslant 1$, and let
\[f:\mathbb{P}^n\longrightarrow X\]
be a surjective map. Then $X\cong \mathbb{P}^n$.
\end{theorem}

\begin{theorem}[\cite{ando}, Theorem 2.3]\label{ando}
Let $f: X\rightarrow Y$ be an elementary contraction on a smooth variety. Assume that the dimension of a fiber of $f$ is at most one. If $\mathrm{dim}f(E)=\mathrm{dim}E-1=n-2$ and $f_E$ is equi-dimensional where E is the exceptional locus of $f$.
Then both $Y$ and $f(E)$ are non-singular, and moreover $f:X\rightarrow Y$ is the blowing-up along the  
smooth center $f(E)$.
\end{theorem}

\begin{theorem}[\cite{ando}, Theorem 3.1]\label{ando2}
Let $f: X\rightarrow Y$ be an elementary contraction on a smooth variety. Assume that the dimension of a fiber of $f$ is at most one. If $\mathrm{dim}Y=n-1$ and $f$ is equi-dimensional then $Y$ is non-singular and $f$ induces a conic bundle structure on $X$.
\end{theorem}

We begin the proof of main theorem.  
\subsection{(2,0)-type}
In this subsection we consider the case where $X$ has a $(2,0)$-type extremal contraction.

\begin{theorem}\label{2,0case}
Let X be a Fano fourfold with $(2,0)$-type contraction. Then $\displaystyle \Lambda ^2 \mathcal{T}_X$ is not nef.
\end{theorem}
To show this, we use the following theorem due to Kawamata.

\begin{theorem}[\cite{kawamata}, Theorem 1.1]\label{kawamata}
Let $X$ be a non-singular projective variety of dimension four defined over $\mathbb{C}$, and let $f:X\rightarrow Y$ be a small elementary contraction. Then the exceptional locus of $f$ is a disjoint union of its irreducible components $E_i$ $(i=1,\cdots,n)$ such that $E_i\cong\mathbb{P}^2$ and $N_{E_i/X}\cong\mathcal{O}_{\mathbb{P}^2}(-1)^{\bigoplus 2}$, where $N_{E_i/X}$ denotes the normal bundle.
\end{theorem}

\begin{proof}[Proof of Theorem \ref{2,0case}]
Let $f:X\rightarrow Y$ be the $(2,0)$-type contraction and $E$ an irreducible component of the exceptional locus of $f$.
Then, by Theorem \ref{kawamata} we have $E\cong\mathbb{P}^2$ and $N_{E/X}\cong\mathcal{O}_{\mathbb{P}^2}(-1)^{\bigoplus 2}$.
From the exact sequence 
\[0\rightarrow \mathcal{T}_E\rightarrow \mathcal{T}_X|_E\rightarrow N_{E/X}\rightarrow 0, \]
we have a surjective homomorphism of vector bundles 
\[\Lambda^2\mathcal{T}_X|_E\rightarrow \Lambda^2N_{E/X}\rightarrow 0.\] 
Since $\Lambda^2N_{E/X}\cong\mathcal{O}_{\mathbb{P}^2}(-2)$ is not nef, $\Lambda ^2 \mathcal{T}_X$ is also not nef.
\end{proof}

\subsection{(3,2)-type}
In this subsection we consider the case where $X$ has a $(3,2)$-type extremal contraction.

\begin{theorem}
Let X be a Fano fourfold with $(3,2)$-type contraction. Then $\displaystyle \Lambda ^2 \mathcal{T}_X$ is not nef.
\end{theorem}

To show this, we use the following theorem due to Ando.

\begin{theorem}[\cite{ando}, Lemma1.5 and Lemma 2.2]\label{3,2}
Let $\varphi:X\rightarrow Z$ be a Fano-Mori contraction of a smooth variety. Suppose that a fiber $F$ of $\varphi$ contains an irreducible component of dimension 1; then
$F$ is of pure dimension 1 and all component of F are smooth rational curves.
If $\varphi$ is birational then $F$ is irreducible and it is a line with respect to $-K_X$.  
\end{theorem}

\begin{proof}
By Theorem \ref{3,2} we have a curve C such that $-K_X.C=1$.
If $\Lambda ^2 \mathcal{T}_X$ is nef, this contradicts to Lemma \ref{2}.
\end{proof}

\begin{lemma}\label{2}
Let X be a smooth n-dimensional projective manifold  with nef vector bundle $\displaystyle \Lambda ^2 \mathcal{T}_X$. 
For any rational curve $C$ in $X$, we have $-K_X.C\geqslant 2$. 
Furthermore, If $\displaystyle \mathcal{T}_X$ is not nef, then $-K_X.C\geqslant n-1$.
\end{lemma}

\begin{proof}
Let $\nu: \mathbb{P}^1\rightarrow C$ be the normalization. We set $\nu^*\mathcal{T}_X\cong\mathcal{O}_{\mathbb{P}^1}(a_1)\oplus\mathcal{O}_{\mathbb{P}^1}(a_2)\oplus\cdots\oplus\mathcal{O}_{\mathbb{P}^1}(a_n)$ where $a_1\geqslant a_2\geqslant\cdots\geqslant a_n$. We have $a_1\geqslant 2$ and $-K_X.C=a_1+a_2+\cdots+a_n$.
Since $\displaystyle \Lambda ^2 \mathcal{T}_X$ is nef, we obtain $a_{n-1}+a_n\geqslant 0$. In particular, we have $a_{n-1}\geqslant 0$. 
Hence we get $-K_X.C=a_1+(a_2+a_3+\cdots+a_{n-2})+(a_{n-1}+a_n)\geqslant 2$.
If $\displaystyle \mathcal{T}_X$ is not nef, then we know that $a_n<0$ and $a_{n-1}\geqslant-a_n>0$.
Therefore, we obtain $-K_X.C\geqslant n-1$.
\end{proof}

\subsection{(3,1)-type}
In this subsection we consider the case where $X$ has a $(3,1)$-type extremal contraction.

\begin{theorem}\label{3,1type}
Let X be a Fano fourfold with $(3,1)$-type contraction. Then $\displaystyle \Lambda ^2 \mathcal{T}_X$ is not nef.
\end{theorem}

To show this, we use the following theorem due to Takagi.

\begin{theorem}[\cite{takagi}, Main Theorem and Theorem 1.1]\label{takagi}
Let X be a smooth fourfold and let $f:X\rightarrow Y$ be a contraction of an extremal ray of $(3,1)$-type.
Let E be the exceptional divisor of f and $C:=f(E)$.
Then
\begin{enumerate}
\item C is a smooth curve $;$
\item $f |_E : E \rightarrow C$ is a $\mathbb{P}^2$-bundle or a quadric bundle over C without multiple fibers $;$
\item f is the blowing-up of Y along C $;$
\item Let F be a general fiber of $f|_E:E\rightarrow C$. Then $(F,-K_X|F)\cong(\mathbb{P}^2,\mathcal{O}_{\mathbb{P}^2}(1)),\;(\mathbb{P}^2,\mathcal{O}_{\mathbb{P}^2}(2))$, 
$(\mathbb{P}^1\times\mathbb{P}^1,\mathcal{O}_{\mathbb{P}^1}(1)\boxtimes\mathcal{O}_{\mathbb{P}^1}(1))$ or $(\mathbb{F}_{2,0},\mathcal{O}_{\mathbb{P}^3}(1)|_{\mathbb{F}_{2,0}})$.
\end{enumerate}
\end{theorem}

\begin{proof}[Proof of Theorem \ref{3,1type}]
Let $f:X\rightarrow Y$ be the $(3,1)$-type contraction, E the exceptional locus of $f$ and $F$ a general fiber of $f|_E$.
Then by Theorem \ref{takagi} we have $(F,-K_X|F)\cong(\mathbb{P}^2,\mathcal{O}_{\mathbb{P}^2}(1)),\;(\mathbb{P}^2,\mathcal{O}_{\mathbb{P}^2}(2))$, 
$(\mathbb{P}^1\times\mathbb{P}^1,\mathcal{O}_{\mathbb{P}^1}(1)\boxtimes\mathcal{O}_{\mathbb{P}^1}(1))$ or $(\mathbb{F}_{2,0},\mathcal{O}_{\mathbb{P}^3}(1)|_{\mathbb{F}_{2,0}})$.
Except for the second case, we have a rational curve $C$ in $F$ such that $-K_X.C=1$. In these cases $\displaystyle \Lambda ^2 \mathcal{T}_X$ is not nef by Lemma \ref{2}. 
If $(F,-K_X|F)\cong(\mathbb{P}^2,\mathcal{O}_{\mathbb{P}^2}(2))$ then X is the blowing-up of smooth manifold Y along smooth curve C. 
Hence $\Lambda ^2 \mathcal{T}_X$ is not nef by Lemma \ref{blowup}. 
\end{proof}

\subsection{(3,0)-type}
In this subsection we consider the case where $X$ has a $(3,0)$-type extremal contraction.

\begin{theorem}\label{3,0type}
Let X be a Fano fourfold with $(3,0)$-type contraction. If $\displaystyle \Lambda ^2 \mathcal{T}_X$ is nef, then X is a blowing-up of the projective space $\mathbb{P}^n$ at a point.
\end{theorem}

To show this, we use following theorems.

\begin{theorem}[\cite{ando}, Theorem 2.1]\label{ando1}
Let $X$ be a smooth projective manifold of dimension n. Let $f:X\rightarrow Y$ be an extremal contraction and $E$ the exceptional locus of $f$. Assume that $\mathrm{dim}E=n-1$.
Let F be a general fiber of $f_E:E\rightarrow f(E)$. Then there exists a Cartier divisor $L$ on $X$ such that 
\begin{enumerate}
\item Im$($Pic$X\rightarrow$Pic$F)=\mathbb{Z}[L|_F]$ and $L|_F$ is ample on $F$ $;$
\item $\mathcal{O}_F(-K_X)\cong\mathcal{O}_F(pL)$ and $\mathcal{O}_F(-E)\cong\mathcal{O}_F(qL)$ for some $p,q\in\mathbb{N}.$ 
\end{enumerate}
\end{theorem}

\begin{theorem}[\cite{beltrametti}, Proposition 2.4 and \cite{fujita}]\label{fujita}
Let X be a smooth 4-fold whose canonical bundle is not nef.
Let $\phi:X\rightarrow Y$ be a contraction of an extremal ray.
If $\phi$ contracts an effective divisor $E$ to a point, then E is isomorphic to either $\mathbb{P}^3$, a $($ possibly singular$)$
hyperquadric or a $($possibly non-normal$)$ Del Pezzo variety. 
\end{theorem}

\begin{theorem}[\cite{kollar}, Theorem 5.14 and Remark 5.15]\label{kollar}
Let X be a projective variety with at most locally complete intersection singularities. Then there is a rational curve $C$ on X such that $-K_X.C\leqslant\dim X+1$.
\end{theorem}

\begin{proof}[Proof of Theorem \ref{3,0type}]
Let $E$ be the exceptional divisor on X. Since $E$ is a Cartier divisor on smooth manifold $X$, $E$ is a locally complete intersection variety.
$E$ is Fano by Theorem \ref{fujita}. 
By Theorem \ref{kollar}, we have a rational curve $C$ in $E$ such that $0<-K_E.C\leqslant4$.
Since 
\[-K_E.C=-K_X.C-E.C\geqslant2+1\geqslant3,\]
we have $-K_E.C=3$ or $4$. 
If $-K_E.C=3$, then we get $-K_X.C=2$ and $-E.C=1$. By Lemma \ref{p-bundle}, this case does not occur. 
Hence we have $-K_E.C=4$. In this case if $-K_X.C=2$, then it is contradiction by Lemma \ref{p-bundle}.
Therefore $-K_X.C=3$, $-E.C=1$. In this case by Theorem \ref{ando1} and Theorem \ref{fujita} we have $E\cong\mathbb{P}^3$ and $N_{E/X}\cong\mathcal{O}_{\mathbb{P}^3}(-1)$.
Therefore X is a blowing-up of smooth manifold Y at a point. By Lemma \ref{blowup} we have $Y\cong\mathbb{P}^4$ and complete the proof.  
\end{proof}

\begin{lemma}\label{p-bundle}
Let X be a smooth variety with nef vector bundle $\Lambda^2\mathcal{T}_X$.
Let C be an extremal rational curve in X such that $-K_X.C=2$.
Then the contraction map $\varphi:X\rightarrow Y$ is $\mathbb{P}^1$-bundle over a smooth variety $Y$ and C is in a fiber of $\varphi$.  
\end{lemma}

\begin{proof}
The proof is based on the proof of Theorem 8 in \cite{mori}. 
Let $\nu:\mathbb{P}^1\rightarrow C\subset X$ be the normalization.
Let $H$ be an irreducible component of the Hom scheme $Hom(\mathbb{P}^1,X)$ containing the morphism $\nu$.
In this case $h^*\mathcal{T}_X\cong\mathcal{O}_{\mathbb{P}^1}(2)\oplus\mathcal{O}^3_{\mathbb{P}^1}$ for 
all $h\in H$ by Lemma \ref{restriction}. 
Then $\dim H=n+2$ and $H$ is smooth.
Let $G$ be the Aut($\mathbb{P}^1$).
Since the natural action of $G$ on $Hom(\mathbb{P}^1,X)$ induces the free action $\sigma$ of $G$ on the connected component containing $H$ and, consequently,
on $H$: 
\[\sigma:G\times H\longrightarrow H, \sigma(g,h)x=h(g^{-1}x), g\in G,h\in H,x\in\mathbb{P}^1, \] 
$G$ also acts on $H\times\mathbb{P}^1$ as follows:
\[\tau:G\times H\times\mathbb{P}^1\longrightarrow H\times\mathbb{P}^1, \tau(g,h,x)=(\sigma(g,h),gx), g\in G,h\in H,x\in\mathbb{P}^1. \] 
Let $Chow^dX$ be the Chow variety parametrizing $1$-dimensional effective cycles $C$ of $X$ with $(C.-K_X)=d$.
Then we have a morphism $\alpha:H\longrightarrow Chow^2X$.
Let Y be the normalization $Y\longrightarrow \overline{\alpha(H)}$ of the closure $\overline{\alpha(H)}$ of $\alpha(H)(\subset Chow^2X)$ in the field 
$k(H)^{G}$ of the $G$-invariant rational function on $H$.
Then $(Y,\Gamma)$ is the geometric quotient $\Gamma:H\rightarrow Y$ of $H$ by $G$ in the sense of Mumford and $Y$ is smooth projective. 
We have also a $G$-invariant morphism:
\[F:H\times\mathbb{P}^1\longrightarrow Y\times X, F(h,x)=(\Gamma(h), h(x)), h\in H,x\in\mathbb{P}^1. \] 
 Let $Z=Spec_{Y\times X}[(F_{\ast}\mathcal{O}_{H\times\mathbb{P}^1})^{G}]$.
 Then $Z$ is the geometric quotient $H\times\mathbb{P}^1/G$ and is a $\mathbb{P}^1$-bundle $q:Z\rightarrow Y$ in the \'etale topology, $\dim Z=n$.
 Moreover the natural projection $p:Z\rightarrow X$ is a smooth morphism from $n$-dimensional smooth variety to $n$-dimensional Fano variety.
 Therefore $p$ is isomorphism.  
\end{proof}

\begin{lemma}\label{restriction}
Let $X$ be a Fano fourfold with nef vector bundle $\displaystyle \Lambda ^2 \mathcal{T}_X$. Let $C$ be a rational curve in $X$ and $\nu:\mathbb{P}^1\rightarrow C$ the normalization.
If $-K_X.C=2$, then $\nu^*\mathcal{T}_X\cong\mathcal{O}_{\mathbb{P}^1}(2)\oplus\mathcal{O}^3_{\mathbb{P}^1}$.
\end{lemma}
\begin{proof}
By easy computation.
\end{proof}

\subsection{(4,1)-type}In this subsection we consider the case where X has only $(4,1)$-type contraction.

\begin{theorem}\label{4,1type}
Let X be a Fano fourfold with only fiber type contraction and having a $(4,1)$-type contraction. If $\displaystyle \Lambda ^2 \mathcal{T}_X$ is nef, then X is isomorphic to the product of the projective line $\mathbb{P}^1$ 
and smooth projective threefold Y with nef tangent bundle $\mathcal{T}_Y$.
\end{theorem}

%To show this, we use the following theorem.

\begin{proof}[Proof of Theorem \ref{4,1type}]
Let $f:X\rightarrow C$ be a $(4,1)$-type contraction.
Then $C$ is isomorphic to $\mathbb{P}^1$.
Let $\rho:X\rightarrow Y$ be another extremal contraction map of extremal ray.
Then the dimension of a non-trivial fiber of $\rho$ is one.
By Theorem \ref{ando}, we know that $\rho$ is a conic bundle.
By Lemma \ref{2} , $\rho$ is a $\mathbb{P}^1$-bundle.
Moreover Y is a Fano threefold by \cite{kollar1}.
We can show that a fiber of $f$ is the section of $\rho$. 
Therefore we have a rank 2 vector bundle $\mathcal{E}$ on $Y$
such that $\rho$ is the natural projection $\rho:X\cong\mathbb{P}_Y(\mathcal{E})\rightarrow Y$.
Take a rational curve $C'$ in $Y$ and put $\nu:\mathbb{P}^1\rightarrow C'$ its normalization.
 In this case we have a morphism of fiber type of Hirzebruch surface $f|_S:S\cong\mathbb{P}_{\mathbb{P}^1}(\nu^*\mathcal{E})\rightarrow C$.  
Therefore $\mathbb{P}_{\mathbb{P}^1}(\nu^*\mathcal{E})\cong\mathbb{P}^1\times\mathbb{P}^1$.
By theorem \ref{biswas}, we have that $\mathcal{E}$ is trivial up to twist.
Therefore we get $X\cong Y\times\mathbb{P}^1$.
We put projections $p:X\rightarrow Y$ and $q:X\rightarrow \mathbb{P}^1$.
By the nefness of \[\Lambda^2\mathcal{T}_X\cong\Lambda^2(p^*\mathcal{T}_Y)\bigoplus (p^*\mathcal{T}_Y\bigotimes q^*\mathcal{T}_{\mathbb{P}^1}),\]
we have the tangent bundle $\mathcal{T}_Y$ is nef. 
\end{proof}

\begin{theorem}[\cite{biswas}, Theorem 1.1]\label{biswas}
Let $\mathcal{E}$ be a vector bundle of rank r over a rationally connected smooth projective variety $X$ over $\mathbb{C}$ such that for every morphism
\[\gamma:\mathbb{P}^1\longrightarrow X,\]
the pullback $\gamma^*\mathcal{E}$ is isomorphic to $\mathcal{L}(\gamma)^{\oplus r}$
for some line bundle $\mathcal{L}(\gamma)^{\oplus r}$ on $\mathbb{P}^1$. 
Then there is a line bundle $\zeta$ over $X$ such that $\mathcal{E}=\zeta^{\oplus r}$.
\end{theorem}
%%%%%%%%%%%%%%%%%%%%%%%%%%%%%%%%%%%%%%%%%%%%%%%%%%%%%%%%%
%%%%%%%%%%%%%%%%%%%%%%%%%%%%%%%%%%%%%%%%%%%%%%%%%%%%%%%%%
%%%%%%%%%%%%%%%%%%%%%%%%%%%%%%%%%%%%%%%%%%%%%%%%%%%%%%%%%%
\subsection{(4,2)-type}
In this subsection we consider the case where X has only $(4,2)$-type and $(4,3)$-type contractions.
If X has a $(4,3)$-type contraction, we can show that it is $\mathbb{P}^1$-bundle. To show this, we use the following theorem.
\begin{theorem}[\cite{kachi}, Theorem 0.6]\label{kachi}
Let $X$ be a $4$-dimensional smooth projective variety, $R$ an extremal ray of $X$, and $g:X\rightarrow Y$ the contraction morphism associated to R.
Assume that $\dim Y=3$. Let $E$ be a $2$-dimensional fiber of $g$.
we call \[l_E(R):=min\{(-K_X.C) | C\;is \;an \;irreducible \;rational\; curve\; contained \;in \;E\}\] the length of $R$ at $E$. It is seen that $l_E(R)=1$ or $2$.
Assume $l_E(R)=2$. Then $E$ is irreducible and is isomorphic to $\mathbb{P}^2$, $N_{E/X}\cong\Omega^1_{\mathbb{P}^2}(1)$ and Y is smooth at P.
\end{theorem}
\begin{lemma}\label{bdle}
Let X be a Fano fourfold with $(4,3)$-type contraction $\varphi: X\rightarrow Y$. 
We assume that $\displaystyle \Lambda ^2 \mathcal{T}_X$ is nef. Then $\varphi$ is $\mathbb{P}^1$-bundle.
\end{lemma}

\begin{proof}
If $\varphi$ has a two dimensional fiber $F$, then by Theorem \ref{kachi} we have $F\cong\mathbb{P}^2$ and $N_{F/X}\cong\Omega^1_{\mathbb{P}^2}(1)$.
We get an exact sequence of vector bundles 
\[\displaystyle \mathcal{T}_X|_F\rightarrow N_{F/X}\cong\Omega^1_{\mathbb{P}^2}(1)\rightarrow 0.\]
Therefore we have the surjection 
\[\displaystyle \Lambda ^2 \mathcal{T}_X|_F\rightarrow  \Lambda ^2\Omega^1_{\mathbb{P}^2}(1)\cong\mathcal{O}_{\mathbb{P}^2}(-1)\rightarrow0.\]
This contradicts to the nefness of $\displaystyle \Lambda ^2 \mathcal{T}_X$.
Hence $\varphi$ is equidimensional. From Theorem \ref{ando2} and Lemma \ref{2}, we know that $\varphi$ is $\mathbb{P}^1$-bundle.
\end{proof}

\begin{theorem}
Let X be a Fano fourfold with only $(4,2)$-type and $(4,3)$-type contractions. We assume that $X$ has a $(4,2)$-type contraction and 
$\displaystyle \Lambda ^2 \mathcal{T}_X$ is nef. Then $\mathcal{T}_X$ is nef.
\end{theorem}
\begin{proof}
Let $f:X\rightarrow S$ be a $(4,2)$-type contraction.
Then $f$ is equi-dimensional and $S$ is smooth. 
By lemma \ref{2} and \ref{p-bundle} a general fiber of $f$ is isomorphic to $\mathbb{P}^2$ and $l(R)=3$ where $R$ is the extremal ray associated with $f$.
Due to Theorem \ref{horing}, $f$ is $\mathbb{P}^2$-bundle and $S$ is a del Pezzo surface by \cite{kollar1}.
Therefore, from Theorem \ref{bundle}, there is a rank 3 vector bundle $\mathcal{E}$ on $S$ such that $f$ is the natural projection 
$X\cong \mathbb{P}_S(\mathcal{E})\rightarrow S$.
Since S cannot have a birational type contraction, $S$ is isomorphic to $\mathbb{P}^2$ or $\mathbb{P}^1\times\mathbb{P}^1$.
If X has another $(4,2)$-type contraction then $X$ is isomorphic to $\mathbb{P}^2\times\mathbb{P}^2$ by Theorem \ref{occhetta}.
If X has two $(4,3)$-type contractions, $X$ is isomorphic to $\mathbb{P}^2\times\mathbb{P}^1\times\mathbb{P}^1$ by Theorem \ref{occhetta}. 
If X has only one $(4,3)$-type contraction, then we have $\rho(X)=2$.
Hence S is $\mathbb{P}^2$.
By the list of rank 3 Fano bundle on $\mathbb{P}^2$, Theorem \ref{wisniewski3}, such a case does not exist.
Therefore we have a proof.  
\end{proof}

\begin{theorem}[\cite{horing}, Theorem 1.3]\label{horing}
Let $X$ be a quasi-projective manifolds that admits an elementary contraction of fibre type $\varphi:X\rightarrow Y$ onto a normal variety $Y$ such that the general fibre has dimension $d$. Suppose that the contraction has length $l(R)=d+1$. If $\varphi$ is equidimensional, it is a projective bundle. 
\end{theorem}

\begin{proposition}[\cite{watanabe}, Proposition 2.5]\label{bundle}
Let $f:X\rightarrow Y$ be a smooth morphism from an $n$-dimensional Fano manifold $X$. If every fiber of f is $\mathbb{P}^d$ and $Y$ is rational, then there exists a rank $(d+1)$ vector bundle $\mathcal{E}$ on $Y$ such that $X=\mathbb{P}_Y(\mathcal{E})$.
\end{proposition}

%\begin{theorem}[\cite{wisniewski1}, Theorem 2.2]\label{wis}
%Let a manifold $X$ of dimension $n$ admit $k$ different contractions of different extremal rays.
%If by $m_i$, $i=1,2,\cdots,k,$ we denote dimensions of images of these contractions, then
%\[\displaystyle \sum^{k}_{i=1}(n-m_i)\leqslant n.\]   
%\end{theorem}

\begin{theorem}[\cite{occhetta}, Theorem1.1]\label{occhetta}
A smooth complex projective variety X of dimension n is isomorphic to a product of projective spaces $\mathbb{P}^{n(1)}\times\cdots\times\mathbb{P}^{n(k)}$
if and only if there exist k unsplit covering families of rational curves $V^1,\cdots,V^k$ of degree $n(1)+1,\cdots,n(k)+1$ with $\sum n(i)=n$
such that the numerical classes of $V^1,\cdots,V^k$ are linearly independent in $N_1(X)$.
\end{theorem}

\subsection{(4,3)-type}
In this subsection we consider the case where X has only $(4,3)$-type contraction.
\begin{theorem}
Let X be a Fano fourfold with only $(4,3)$-type contraction. 
If $\displaystyle \Lambda ^2 \mathcal{T}_X$ is nef, then $\mathcal{T}_X$ is nef.
\end{theorem}
\begin{proof}
We note that, in this case, any $(4,3)$-type contraction is $\mathbb{P}^1$-bundle by Lemma \ref{bdle}.
If X has four $(4,3)$-type contractions then $X$ is isomorphic to $\mathbb{P}^1\times\mathbb{P}^1\times\mathbb{P}^1\times\mathbb{P}^1$ by Theorem \ref{occhetta} and lemma \ref{2}.

If $X$ has three $(4,3)$-type contractions then $\rho(X)=3$.
Let $\phi: X\rightarrow Z$ be a $\mathbb{P}^1$-bundle one of the above contractions. 
Then, from \cite{kollar1}, $Z$ is a smooth Fano threefold Z with $\rho(Z)=2$.
Moreover, by the surjection of vector bundles \[\displaystyle \Lambda ^2 \mathcal{T}_X\rightarrow \pi^*(\Lambda ^2 \mathcal{T}_Z)\rightarrow 0,\] we know that 
$\Lambda ^2 \mathcal{T}_Z$ is nef.
Due to Theorem \ref{3-dim} and \ref{3-nef}, $Z$ is rational. 
From Theorem \ref{bundle} there is a rank 2 vector bundle $\mathcal{E}$ on Z such that $X$ is isomorphic to $\pi:\mathbb{P}_Z(\mathcal{E})\rightarrow Z$.
Since X does not have a birational type contraction, we have $Z$ is either $\mathbb{P}^1\times\mathbb{P}^2$ or $\mathbb{P}_{\mathbb{P}^2}(\mathcal{T}_{\mathbb{P}^2})$
by theorem \ref{3-dim} and theorem \ref{3-nef}.
If Z is $\mathbb{P}^1\times\mathbb{P}^2$ $($ resp., $\mathbb{P}_{\mathbb{P}^2}(\mathcal{T}_{\mathbb{P}^2})$ $)$, we have $\mathcal{E}|_C\cong\mathcal{O}^2_{\mathbb{P}^1}(a)$ 
for every fiber $C\cong\mathbb{P}^1$ of natural projection $\mathbb{P}^1\times\mathbb{P}^2\rightarrow \mathbb{P}^2$ $($ resp., $\mathbb{P}_{\mathbb{P}^2}(\mathcal{T}_{\mathbb{P}^2})\rightarrow \mathbb{P}^2$ $)$ since 
\[2=-K_Z.C=-K_X.\tilde{C}+(a_2-a_1)\geqslant 2+(a_2-a_1)\]
 where $\tilde{C}$ is an extremal rational curve on X associated with C and $\mathcal{E}|_C\cong\mathcal{O}_{\mathbb{P}^1}(a_1)\oplus\mathcal{O}_{\mathbb{P}^1}(a_2)$ $($ $a_1\leqslant a_2$ $)$. 
 Therefore there is a rank 2 Fano bundle $\mathcal{E}'$ on $\mathbb{P}^2$ such that $\mathcal{E}$ is the pullback of $\mathcal{E}'$ by the natural projection $\pi:Z\rightarrow\mathbb{P}^2$ up to twist by a line bundle.   
$X$ is isomorphic to $\mathbb{P}^1\times\mathbb{P}_{\mathbb{P}^2}(\mathcal{T}_{\mathbb{P}^2})$ by using the list of Fano bundles in Theorem \ref{wisniewski2}.

If X has two $(4,3)$-type contraction then $\rho(X)=2$.
Let $\phi : X\rightarrow Y$ be a $\mathbb{P}^1$-bundle which is a one of above contractions.
From the surjection $\Lambda^2\mathcal{T}_X\rightarrow \phi^*(\Lambda^2\mathcal{T}_Y)\rightarrow 0$, we know that $Y$ is a Fano threefold of $\rho(Y)=1$ such that 
$\Lambda^2\mathcal{T}_Y$ is nef.
If $\mathcal{T}_Y$ is not nef, then there is a line $l$ in $Y$ such that $\mathcal{T}_Y$ is not nef on $l$ by the argument in the proof of \cite{peternell1} Theorem 5.1.
We know that $-K_Y.l=2$ by Theorem \ref{3-dim}.
Therefore we can write $\mathcal{T}_Y|_l\cong \mathcal{O}_l(2)\oplus\mathcal{O}_l(k)\oplus\mathcal{O}_l(-k)$ for some positive integer $k$. 
Let $F$ be the preimage of $l$ by $\phi$. 
Then, F is isomorphic to a Hirzebruch surface $\pi: \mathbb{P}_{l}(\mathcal{O}_l(a)\oplus\mathcal{O}_l(b))\rightarrow l$ for some integer $a$, $b$ such that $a\geqslant b$.
 Put $C$ a section of $\pi$ which is associated with $\mathcal{O}_l(b)$.
 Then, we have $2=-K_Y.l=-K_X.C+(a-b)\geqslant 2+0=2$.
 Therefore we get $-K_X.C=2$ and $a=b$.
 On the other hand, we have the exact sequence 
 \[0\rightarrow \mathcal{T}_{\phi}\rightarrow \mathcal{T}_X\rightarrow \phi^*\mathcal{T}_{Y}\rightarrow 0. \]
 Restrict to C, we have 
 \[0\rightarrow \mathcal{O}_C\rightarrow \mathcal{T}_X|_C\rightarrow \mathcal{O}_C(2)\oplus\mathcal{O}_C(k)\oplus\mathcal{O}_C(-k)\rightarrow 0. \]
 Hence $\mathcal{T}_X$ is not nef on $C$.
 If $\Lambda^2\mathcal{T}_X$ is nef, we have $-K_X.C\geqslant3$ by Lemma \ref{2}.
 It is a contradiction.
 Therefore we see that $\mathcal{T}_Y$ is nef, that is, $Y$ is isomorphic to the $\mathbb{P}^3$ or a three dimensional smooth quadric hypersurface.
  Moreover, $\phi$ is a projectivization of rank 2 vector bundle on $Y$ by Theorem \ref{bundle}.
  By the classification of raank 2 Fano bundle $($see Theorems \ref{wisniewski4} and \ref{wisniewski5}$)$, we know that $Y$ is $\mathbb{P}^3$ and 
  such a bundle is the null correlation bundle.
  In particular, $\mathcal{T}_X$ itself is nef.
\end{proof}

%%%%%%%%%%%%%%%%%%%%%%%%%%%%%%%%%%%%%%%%%%
\section{Varieties with nef tangent bundle}
In this section we list results about varieties whose tangent bundles are nef.
 %%%%%%%%%%%%%%%%%%%%%%%%%%%%%%%%%%%%%%%%%%%%%%%%%
 \subsection{Surfaces with $\mathcal{T}_X$ nef}
\begin{theorem}[\cite{peternell1}, Theorem3.1] 
Let X be a smooth projective surface and assume $\mathcal{T}_X$ to be nef.
Then $X$ is minimal and exactly one of the surfaces in the following list.
\begin{enumerate}
\item $X=\mathbb{P}^2$.
\item $X=\mathbb{P}^1\times\mathbb{P}^1$.
\item $X=\mathbb{P}(\mathcal{E})$, $\mathcal{E}$ a semistable rank 2-vector bundle on an elliptic curve $C$.
\item $X$ is abelian surface. 
\item $X$ is hyperelliptic.
\end{enumerate} 
\end{theorem}
%%%%%%%%%%%%%%%%%%%%%%%%%%%%%%%%%%%%%%%%%%%%%%%%
 \subsection{3-folds with $\mathcal{T}_X$ nef}
\begin{theorem}[\cite{peternell1}, Theorem 6.1 and Theorem10.1] \label{3-nef}
Let X be a smooth projective threefold and assume $\mathcal{T}_X$ to be nef.
Then some \'etale covering $\tilde{X}$ of $X$ belong to the following list.
\begin{enumerate}
\item $\mathbb{P}^3$.
\item 3-dimensional smooth quadric $Q_3$.
\item $\mathbb{P}^1\times\mathbb{P}^2$.
\item $\mathbb{P}(\mathcal{T}_{\mathbb{P}^2})$.
\item $\mathbb{P}^1\times\mathbb{P}^1\times\mathbb{P}^1$.
\item $\tilde{X}=\mathbb{P}(\mathcal{E})$ for a flat rank 3-vector bundle on an elliptic curve $C$.
\item $X=\mathbb{P}(\mathcal{F})\times_C\mathbb{P}(\mathcal{F}')$ for flat rank 2-vector bundles $\mathcal{F}$ and $\mathcal{F}'$ over an elliptic curve $C$.
\item $\tilde{X}=\mathbb{P}(\mathcal{E})$ for a flat rank 2-vector bundle $\mathcal{E}$ on an abelian surface.
\item $\tilde{X}$ is an abelian threefold. 
\end{enumerate} 
\end{theorem}
%%%%%%%%%%%%%%%%%%%%%%%%%%%%%%%%%%%%%%%%%%%%%%%%%%%%%%%%%%%%%%%%%%%%%%%%%%%%%%%%%%%%%%%%%%%%%%%%%%%%%%%%%%%%%%%%%%%
 \subsection{4-folds with $\mathcal{T}_X$ nef}
\begin{theorem}[\cite{peternell3}, Theorem 3.1, \cite{mok} Main Theorem and \cite{hwang} Theorem 4.2] 
Let X be a smooth 4-fold with $\mathcal{T}_X$ nef.
Then some \'etale covering $\tilde{X}$ of $X$ belong to the following list.
\begin{enumerate}
\item $\mathbb{P}^4$. 
\item 4-dimensional smooth quadric $Q_4$.
\item $\mathbb{P}^3\times\mathbb{P}^1$. 
\item $Q_3\times\mathbb{P}^1$.
\item $\mathbb{P}^2\times\mathbb{P}^2$.
\item $\mathbb{P}^1\times\mathbb{P}^1\times\mathbb{P}^2$.
\item $\mathbb{P}(\mathcal{T}_{\mathbb{P}^2})\times\mathbb{P}^1$.
\item $\mathbb{P}^1\times\mathbb{P}^1\times\mathbb{P}^1\times\mathbb{P}^1$.
\item $\mathbb{P}(\mathcal{E})$ with a null correlation bundle $\mathcal{E}$ on $\mathbb{P}^3$.
\item a $Q_3$- or a flat $\mathbb{P}^3$-bundle over an elliptic curve $C$.
\item a flat $\mathbb{P}^2$-bundle over a ruled surface Y over an elliptic curve with $\mathcal{T}_Y$ nef.
\item a flat $\mathbb{P}^1$-bundle over a flat $\mathbb{P}^2$- or $\mathbb{P}^1\times\mathbb{P}^1$-bundle over an elliptic curve with $\mathcal{T}_Y$ nef.
\item $\tilde{X}=\mathbb{P}(\mathcal{E})$ for a flat rank 3-vector bundle $\mathcal{E}$ on an abelian surface.
\item $X=\mathbb{P}(\mathcal{F})\times_A\mathbb{P}(\mathcal{F}')$ for flat rank 2-vector bundles $\mathcal{F}$ and $\mathcal{F}'$ over an abelian surface $A$.
\item $\tilde{X}=\mathbb{P}(\mathcal{E})$ for a flat rank 2-vector bundle $\mathcal{E}$ on an abelian 3-fold.
\item abelian 4-fold.
\end{enumerate} 
\end{theorem}
%%%%%%%%%%%%%%%%%%%%%%%%%%%%%%%%%%%%%%%%%%%%%%
%%%%%%%%%%%%%%%%%%%%%%%%%%%%%%%%%%%%%%%%%%%%%%%
\section{Fano bundle}
In this section we collect results about Fano bundles (i.e. vector bundles whose projectivization are Fano) used
in the proof of main theorem. 
%%%%%%%%%%%%%%%%%%%%%%%%%%%%%%%%%%%%%%%%%%%%%%
\begin{theorem}[\cite{wisniewski2}, Theorem]\label{wisniewski2}
 $\mathcal{E}$ be a rank 2 Fano bundle on $\mathbb{P}^2$. 
Then $\mathcal{E}$ is isomorphic to one of the following up to twist by some line bundle:
 \begin{enumerate}
\item $\mathcal{O}_{\mathbb{P}^2}(1)\bigoplus\mathcal{O}_{\mathbb{P}^2}(-1)$
\item $\mathcal{O}_{\mathbb{P}^2}(1)\bigoplus\mathcal{O}_{\mathbb{P}^2}$
\item $\mathcal{O}_{\mathbb{P}^2}\bigoplus\mathcal{O}_{\mathbb{P}^2}$
\item $\mathcal{T}_{\mathbb{P}^2}$
\item $0\rightarrow\mathcal{O}_{\mathbb{P}^2}\rightarrow\mathcal{E}\rightarrow \mathcal{I}_x\rightarrow 0$ where $\mathcal{I}_x$ is the ideal sheaf of a point
\item stable bundle with $c_1=0$, $c_2=2$ 
\item stable bundle with $c_1=0$, $c_2=3$ and $\mathcal{E}(1)$ is spanned
\end{enumerate}
\end{theorem}
%%%%%%%%%%%%%%%%%%%%%%%%%%%%%%%%%%%%%%%%%%%%%%%%%
\begin{theorem}[\cite{wisniewski3}, Theorem]\label{wisniewski3}
Let $\mathcal{E}$ be a rank 3 Fano bundle on $\mathbb{P}^2$. 
Then $\mathcal{E}$ is isomorphic to one of the following up to twist by some line bundle:
 \begin{enumerate}
\item $\mathcal{O}_{\mathbb{P}^2}^3$
\item $\mathcal{O}_{\mathbb{P}^2}(1)\bigoplus\mathcal{O}_{\mathbb{P}^2}^2$
\item $\mathcal{T}_{\mathbb{P}^2}(-1)\bigoplus\mathcal{O}_{\mathbb{P}^2}$
\item $\mathcal{O}_{\mathbb{P}^2}(2)\bigoplus\mathcal{O}_{\mathbb{P}^2}^2$
\item $\mathcal{O}_{\mathbb{P}^2}^2(1)\bigoplus\mathcal{O}_{\mathbb{P}^2}$
\item $\mathcal{T}_{\mathbb{P}^2}(-1)\bigoplus\mathcal{O}_{\mathbb{P}^2}(1)$
\item $\mathcal{O}_{\mathbb{P}^2}\bigoplus\mathcal{E}_2$ where $\mathcal{E}_2$ is in 
$0\rightarrow\mathcal{O}_{\mathbb{P}^2}\rightarrow\mathcal{E}_2(-1)\rightarrow \mathcal{I}_x\rightarrow 0$, $\mathcal{I}_x$ is the ideal sheaf of a point
\item a bundle fitting into $0\rightarrow\mathcal{O}_{\mathbb{P}^2}^2(-1)\rightarrow\mathcal{O}_{\mathbb{P}^2}^5\rightarrow \mathcal{E}\rightarrow 0$ 
\item a bundle fitting into $0\rightarrow\mathcal{O}_{\mathbb{P}^2}(-2)\rightarrow\mathcal{O}_{\mathbb{P}^2}^4\rightarrow \mathcal{E}\rightarrow 0$ 
\end{enumerate}
\end{theorem}
%%%%%%%%%%%%%%%%%%%%%%%%%%%%%%%%%%%%%%%%%%%%%%%%%%
\begin{theorem}[\cite{wisniewski4}, Theorem 2.1]\label{wisniewski4}
Let $\mathcal{E}$ be a rank 2 Fano bundle on $\mathbb{P}^3$. 
Then $\mathcal{E}$ is isomorphic to one of the following up to twist by some line bundle:
 \begin{enumerate}
\item $\mathcal{O}_{\mathbb{P}^3}\bigoplus\mathcal{O}_{\mathbb{P}^3}$
\item $\mathcal{O}_{\mathbb{P}^3}\bigoplus\mathcal{O}_{\mathbb{P}^3}(-1)$
\item $\mathcal{O}_{\mathbb{P}^3}(-1)\bigoplus\mathcal{O}_{\mathbb{P}^3}(1)$
\item $\mathcal{O}_{\mathbb{P}^3}(-2)\bigoplus\mathcal{O}_{\mathbb{P}^3}(1)$
\item the null-correlation bundle $\mathcal{N}$
\end{enumerate}
\end{theorem}
%%%%%%%%%%%%%%%%%%%%%%%%%%%%%%%%%%%%%%%%%%%%%%%%%
\begin{theorem}[\cite{wisniewski5}, Theorem]\label{wisniewski5}
Let $\mathcal{E}$ be a rank 2 Fano bundle on $Q_3$. 
Then $\mathcal{E}$ is isomorphic to one of the following up to twist by some line bundle:
 \begin{enumerate}
\item $\mathcal{O}_{Q_3}\bigoplus\mathcal{O}_{Q_3}(-1)$
\item spinor bundle
\item $\mathcal{O}_{Q_3}\bigoplus\mathcal{O}_{Q_3}$
\item $\mathcal{O}_{Q_3}(-1)\bigoplus\mathcal{O}_{Q_3}(1)$
\item stable bundle with $c_1=0$, $c_2=2$ 
\end{enumerate}
\end{theorem}
%%%%%%%%%%%%%%%%%%%%%%%%%%%%%%%%%%%%%%%%%%%%%%%%%%
%%%%%%%%%%%%%%%%%%%%%%%%%%%%%%%%%%%%%%%%%%%%%%%%%

\end{document}